\documentclass{amsart}

\usepackage{amssymb,amscd}
\usepackage{amsmath,latexsym,amssymb, amscd}
\usepackage{amssymb,latexsym, xspace, enumerate}

\usepackage{graphicx}

\newtheorem{theorem}{Theorem}[section]
\newtheorem{lemma}[theorem]{Lemma}
\newtheorem{proposition}[theorem]{Proposition}
\newtheorem{corollary}[theorem]{Corollary}

\newtheorem{claim}[theorem]{Claim}

\long\def\alert#1{\smallskip{\hskip\parindent\vrule%
\vbox{\advance\hsize-2\parindent\hrule\smallskip\parindent.4\parindent%
\narrower\noindent#1\smallskip\hrule}\vrule\hfill}\smallskip}

\theoremstyle{definition}
\newtheorem{definition}[theorem]{Definition}
\newtheorem{definitions}[theorem]{Definitions}
\newtheorem{example}[theorem]{Example}

\newtheorem{remark}[theorem]{Remark}
\newtheorem{question}[theorem]{Question}
\newtheorem{questions}[theorem]{Questions}

\newtheorem{setting}[theorem]{Setting}
\newtheorem{settingnotation}[theorem]{Setting and Notation}
\numberwithin{equation}{section}

\newtheorem{remark/questions}[theorem]{Remark and Questions}

\newtheorem{remarks}[theorem]{Remarks}

\newtheorem{notationremarks}[theorem]{Notation and Remarks}

 \long\def\alert#1{\smallskip\line{\hskip\parindent\vrule
\vbox{\advance\hsize-2\parindent\hrule\smallskip\parindent.4\parindent
  \narrower\noindent#1\smallskip\hrule}\vrule\hfill}\smallskip}

\def\subheading#1{\noindent{\bf{#1}}\vskip12pt}
\def\frac#1#2{{{#1}\over{#2}}}

\def\endresult#1{\medskip}

\def \Q{{{\mathbb Q }}}
  \def \N{{\mathbb  N}}

\def\q{{\bf q}}
\def\p{{\bf p}}
\def\m{{\bf m}}
\def\n{{\bf n}}

  \def \w{{ \mbox{\bf w}}}

\def\dim{\mathop{\rm dim}}

\def\Spec{\mathop{\rm Spec}}

\def\hgt{\mathop{\rm ht}}

\begin{document}
\baselineskip 15 pt

\title[Excellent normal domains  and Krull domains]{
 Excellent normal local domains and \\ extensions of Krull domains}
\author{William Heinzer}
\address{Department of Mathematics, Purdue University, West
Lafayette, Indiana 47907}

\email{heinzer@math.purdue.edu}

\author{Christel Rotthaus}
\address{Department of Mathematics, Michigan State University, East
  Lansing,
   MI 48824-1027}

\email{rotthaus@math.msu.edu}

\author{Sylvia Wiegand}
\address{Department of Mathematics, University of Nebraska, Lincoln,
NE 68588-0130}

\email{swiegand@math.unl.edu}


\dedicatory{This paper is dedicated to Hans-Bj\o rn Foxby.}


   \thanks{The authors are grateful for
  the hospitality, cooperation and support of Michigan State,
  Nebraska,
   Purdue and the CIRM at Luminy, France, where several work sessions on related research were
   conducted.  We thank the referee for a careful reading of the paper and helpful comments.}

   \begin{abstract}
We consider properties of extensions of  Krull domains  such as flatness
that  involve behavior of extensions and
contractions of prime ideals. Let $(R,\m)$ be an excellent normal local domain
with field of fractions $K$, let $y$ be a nonzero element of $\m$ and let
$R^*$ denote the $(y)$-adic completion of $R$. For elements  $\tau_1, \ldots, \tau_s$
 of $yR^*$ that are algebraically independent over $R$, we construct
 two associated
Krull domains:  an intersection domain
$A := K(\tau_1, \ldots \tau_s) \cap R^*$ and its  approximation
domain $B$; see Setting~\ref{7.4.1}.

If in addition $R$ is countable with $\dim R \ge 2$,
we prove that there exist elements $\tau_1, \ldots, \tau_s, \ldots $ as
above such that, for each $s \in \N$, the extension
$R[\tau_1, \ldots, \tau_s] \hookrightarrow R^*[1/y]$ is flat; equivalently,
$B = A$ and $A$ is Noetherian.  Using this result we establish the existence of a
normal Noetherian local domain $B$ such that:  $B$ dominates $R$; $B$ has $(y)$-adic
completion $R^*$; and $B$ contains a height-one prime ideal $\p$ such that
$R^*/\p R^*$ is not reduced. Thus $B$ is not a Nagata domain and hence is not excellent.

We present several theorems involving the construction. These theorems yield examples
where  $B \subsetneq A$ and $A$ is Noetherian while $B$ is not Noetherian; and
other examples where $B = A$ and $A$ is not Noetherian.

\end{abstract}

\maketitle

\section{Introduction}\label{intro}

About twenty years ago Judy Sally gave an expository talk on the following question:

\begin{question} \label{SallyQ} What rings lie between a Noetherian
integral domain $S$ and its field of fractions $\mathcal Q(S)$?
\end{question}

We are inspired by work of Shreeram Abhyankar such as that in his paper \cite{Abhy0}
to
ask the following related question: \footnote{Ram's work demonstrates the vastness of
power series rings.  The authors have fond memories of many pleasant conversations
with him concerning power series.}

\begin{question} \label{RamQ}
Let $I$ be an ideal of a Noetherian integral domain $R$
and let $R^*$ denote the $I$-adic completion of $R$.
What rings lie between  $R$ and $R^*$?
\end{question}

A  wide variety of integral domains fit the descriptions of both
Questions~\ref{SallyQ} and \ref{RamQ}.
Let $(R,\m)$ be an excellent normal local domain and let $S$ be a
polynomial ring in finitely many variables over $R$.
In work over a number of years related to these questions, the authors
have been developing  techniques for constructing examples that are
birational extensions  of $S$ and also subrings of an ideal-adic
completion of $R$.  Classical constructions of Noetherian integral domains with interesting
properties,  such as failure to be a Nagata ring, have been given by  Akizuki, Schmidt,
 Nagata and  others, \cite{A}, \cite{Sc},  \cite{N1}.  We recall that a ring $A$ is a
{\it Nagata ring } if $A$ is Noetherian and if  the integral closure
of $A/P$ in $L$ is finite over $A/P$, for every  prime ideal $P$ of
$A$ and every field $L$ finite algebraic over the field of fractions
of $A/P$, \cite[page 264]{M}.

We often use in our construction the completion of an excellent normal local
domain $(R,\m)$ with respect to a principal ideal $yR$,
where $y$ is a nonzero nonunit  of $R$.  The $(y)$-adic
completion $R^*$ of $R$ may be regarded as either an inverse limit or
as a homomorphic image of a formal power series ring $R[[z]]$ over $R$.
Thus we have
$$
R^* ~=~
\underset{n}{\varprojlim}\phantom{x} \Big(\frac{R}{y^nR}\Big)  ~ = ~ \frac{R[[z]]}{(z-y)R[[z]]}.
$$
An element  $\tau$  of $R^*$ has an expression as a power series in $y$
with coefficients in $R$
It is often the case that there exist elements $\tau_1, \ldots, \tau_n$ in $R^*$
that are algebraically independent over $R$.
An elementary cardinality argument shows this  is always the case if $R$ is countable.
Assume that $\tau_1, \ldots, \tau_n$
in $R^*$ are algebraically independent over $R$. By modifying $\tau_i$ by
an element in $R$, we may assume  that each $\tau_i \in yR^*$.
Let $S := R[\tau_1, \ldots, \tau_n]$.  Then $S$ is both a subring of $R^*$ and
a polynomial ring in $n$ variables over $R$.
Although the expression for the $\tau_i$ as power series in $y$ with coefficients
in $R$ is not unique,  we use it to construct
an integral domain $B$ that is a directed union of localized polynomial rings over $R$.

The construction we consider associates with  $R$ and  $\tau_1, \ldots, \tau_n$
the following two integral domains:
\begin{enumerate}
\item   an intersection domain $A := \mathcal Q(S) \cap R^*$,  and
\item   an integral domain $B \subseteq A$ that approximates $A$.
\end{enumerate}
 The integral domain $B$ is a directed union of localized polynomial rings in $n$
variables over $R$.  The rings $B$ and $A$ are birational extensions of
the polynomial ring $S$ and subrings of $R^*$.
Thus they fit  the descriptions of  both
Questions~\ref{SallyQ} and \ref{RamQ}.
For   integral domains $A$ and $B$ obtained as above,   we ask:

\begin{questions} \label{PropQ}  For a given excellent normal local
domain $(R,\m)$ and elements $\tau_1, \ldots, \tau_n$
as above,   what properties do the constructed rings  $A$ and $B$ have,
 and what criteria determine these properties?
\end{questions}

Our work in this article concerning Questions~\ref{PropQ}  focuses
primarily on the case where the base ring $R$ is an excellent normal local domain.
The intersection  domain $A  = \mathcal Q(S) \cap R^*$ may fail to be Noetherian
even though  $R$,  and therefore $R^*$,   is an excellent normal local domain.
However,  the intersection domain
  $A$ is always a Krull domain,  and the
$(y)$-adic completion of $A$ is $R^*$.
Thus, in order to present an iterative procedure,  in Section~\ref{basicsapprox}
we present many of the properties we study with the following Krull domain setting:

\begin{setting} \label{Krullset} Let $(T,\n)$  be a  local Krull domain
with field of fractions $\mathcal Q(T)$.
Assume that  $y\in \n$ is a nonzero element
such that the $(y)$-adic completion $ (T^*,\n^*)$
of $T$  is an  analytically normal Noetherian local domain.
Since the $\n$-adic
completion of $T$ is the same as the $\n^*$-adic completion
of $T^*$, it  follows that the $\n$-adic completion $\widehat
T$ of $T$ is also a normal Noetherian local domain.
Let $\mathcal Q(T^*)$ denote the field of fractions  of $T^*$.
Since $T^*$ is Noetherian,  $\widehat T$ is faithfully flat over $T^*$ and
we have
$T^* = \widehat T \cap \mathcal Q(T^*)$. Therefore
$\mathcal Q(T)\cap T^*=\mathcal Q(T)\cap\widehat T$.
Assume that $T = \mathcal Q(T) \cap T^*$,  and let
$d$ denote the dimension of
the Noetherian domain $T^*$. It follows that $d$ is also
the dimension of $\widehat T$.

Let $\tau_1\dots,\tau_s$ be elements of $yT^*$ that are algebraically
independent over  $T$. We consider the extensions
$$T  ~\hookrightarrow ~ T[\tau_1\dots,\tau_s]~ \hookrightarrow ~
A~ :=\mathcal Q(T)(\tau_1, \ldots, \tau_s)~\cap ~ T^*~\hookrightarrow ~
T^*.$$
In particular, the following map is critical:

\begin{equation*} \tag{\ref{Krullset}.0}
\varphi: ~~T[\tau_1\dots,\tau_s]~\hookrightarrow ~
T^*[1/y].
\end{equation*}
\end{setting}

The intersection ring $A = \mathcal Q(T)(\tau_1, \ldots, \tau_s)
\cap T^*$ and its approximating ring $B$ may be Noetherian or not,
or excellent or not. Examples~\ref{4.7.13} and \ref{wfexample}
demonstrate that one may have $B = A$, or $B \subsetneq A$.
Properties of the rings $B$ and $A$ are related to properties  of
the map $\varphi$ of Equation~\ref{Krullset}.0; this is illustrated
by two conclusions of Theorem~\ref{7.5.45}:
\begin{enumerate}
\item  $B$ is Noetherian
if and only if $\varphi$ is flat.
\item If $B$ is Noetherian, then  $B = A$.
\end{enumerate}
In general, by Theorem~\ref{7.5.5}, we have  $B = A$  if and only if
$\varphi$ satisfies the  {\it weak flatness } property  in
Definitions~\ref{wfetc} below.
We  describe   in Section~\ref{basicsapprox} the construction of $B$.

The rings $T[\tau_1, \ldots, \tau_s]$ and $T^*[1/y]$ in
Equation~\ref{Krullset}.0 are Krull domains as are also the
constructed rings $B \hookrightarrow A$. We demonstrate connections between properties
of the extension of Krull domains
defined by the map $\varphi$ in Equation~\ref{Krullset}.0, and properties of the Krull domains
$B$ and $A$.
 Some of our results hold for Krull domains as in Setting~\ref{Krullset};
for others we restrict to
the case where $T = R$ is an  excellent normal local domain.

It is useful to give a name for the  elements $\tau_1,\cdots, \tau_s$ in case
  the map   $\varphi$ of Equation~\ref{Krullset}.0 is flat.

\begin{definition}  \label{prlidef}  Assume Setting~\ref{Krullset}; thus  $T$ is a  local Krull
domain and   $y$ is a nonzero nonunit element of
 $T$  such that  the $(y)$-adic completion $T^*$  of $T$ is an analytically normal
Noetherian local domain with $T = \mathcal Q(T) \cap T^*$.
Elements $\tau_1\dots,\tau_s \in yT^*$ that are algebraically independent over  $T$
are said to be {\it
primarily limit-intersecting} in $y$ over $T$ provided
the inclusion map
$$
\varphi : ~ T[\tau_1,\dots,\tau_s] ~ \hookrightarrow ~  T^*[1/y]
$$
is flat.   An infinite set $\{\tau_i\}_{i=1}^\infty$ of elements in $yT^*$ that are algebraically independent over $T$ is
said to be {\it primarily limit-intersecting} in $y$ over $T$ if for each positive integer $s$, the  elements
$\tau_1, \ldots, \tau_s$ are  primarily limit-intersecting in $y$ over $T$.
\end{definition}

It is natural to ask about the existence of primarily limit-intersecting elements:

\begin{question} \label{ques1}
Let $R$ be an excellent normal local domain  with
$\dim R = d \ge 2$, let $y$ be a nonzero element in
the maximal ideal $\m$ of $R$, and let $R^*$ be the $(y)$-adic completion of $R$.
Under what conditions on $R$ do there exist elements
that are primarily limit-intersecting in $y$ over $R$?
\end{question}

With notation as in Question~\ref{ques1},
we describe  in Theorem~\ref{iff} and Remark~\ref{iffremark}
necessary and sufficient conditions that  an element
$\tau  \in yR^*$ be primarily limit-intersecting in $y$ over $R$. If $R$ is countable,
we prove in Theorem~\ref{exnprimli} the existence of an infinite sequence
$\tau_1, \ldots, \tau_s, \ldots \in yR^*$ of elements that are primarily
limit-intersecting in $y$ over $R$. We show in Theorem~\ref{nagtheorem}  that in general
for an element $\eta \in yR^*$ that is
primarily limit-intersecting in $y$ over $R$, the constructed Noetherian
domain
$$
B ~= ~ A ~ =  ~ R^* \cap \mathcal Q(R[\eta])
$$ may fail to be excellent.

In Section~\ref{section4} we  present two theorems involving the construction.
These theorems yield examples
where  $B \subsetneq A$ and $A$ is Noetherian while $B$ is not Noetherian; and
other examples where $B = A$ and $A$ is not Noetherian.  We
describe several examples obtained by iteration of
the construction considered in Section~\ref{section3}.

\section{Basic properties and the approximation domain\label{basicsapprox}}

In this section we give  background and  terminology.
We generally assume Setting~\ref{Krullset} in this section.
First we illustrate the construction of
an intersection domain $A$ and develop the terminology necessary
for an approximation domain $B$ in
 what we consider the easiest
example that can be constructed.

\begin{example}\label{easy}  {\bf The ``easiest" example. } Let $k$ be a field,
for example, $k=\Q$, the rational numbers, and let
$x$ be an indeterminate over $k$. Let $R=k[x]_{(x)}$ and  $R^*=k[[x]]$, the power
series ring in $x$ over $k$ and the $x$-adic completion of $R$.
Let $\tau\in xk[[x]]$ be algebraically independent over $R$; for example,
if $k=\Q$, we could take $\tau=e^x-1$.
Define $A$, the {\it intersection domain} associated to $\tau$ over $R$ by
$$
A  ~=   ~k(x,\tau)\cap k[[x]].
$$

In this case $A$ is a rank-one  discrete valuation ring (DVR) because it is the intersection
of the DVR $k[[x]]$ with a subfield of $k((x))$ that is not contained in $k$.
Thus $A$ is a Noetherian one-dimensional regular local ring (RLR) and the
unique maximal ideal is $xA$.

We apply  approximation techniques to more precisely  describe  the elements
that are in  $A$.  In order to  define an approximation domain $B$ that goes with $A$,
write $$\tau ~= ~\sum_{i=1}^\infty a_ix^i,$$
where the $a_i\in k$. Define $\tau_0=\tau$.
For each $n\in\N,$ define the $n^{\text{th}}$ {\it endpiece} of $\tau$, denoted $\tau_n$,
and define rings $U_n$ and $B_n$ by
$$
\tau_n ~= ~\sum_{i=n+1}^\infty a_ix^{i-n},  \qquad  U_n ~= ~k[x,\tau_n]\qquad
\text{ and } \qquad B_n~= ~k[x,\tau_n]_{(x,\tau_n)}.
$$
Set $$U ~= ~\cup _{n=0}^\infty U_n  \qquad \text{ and } \qquad B~=~\cup _{n=0}^\infty B_n.$$
\end{example}

It is straightforward to show that  $A = B$ in Example~\ref{easy}; see \cite[Chapter~6]{powerbook}.
Since the extension $k[x, \tau_n] \hookrightarrow A$ does not satisfy the
dimension inequality \cite[p. 119]{M},
the ring $A=B$ is {\it not} the localization of a finitely generated algebra over~$k$.

In general the intersection of a normal Noetherian domain with a
subfield of its field of fractions is a Krull domain, but need  not  be
Noetherian.  A directed union of normal Noetherian domains
may be a non-Noetherian Krull domain. Thus, in order to be able to iterate
our construction, we consider a local
Krull domain $(T,\n)$ that  is not assumed to be Noetherian, but is
assumed to have a Noetherian completion. To distinguish from the
 Noetherian hypothesis on $R$, we let  $T$ denote the base
domain.

\bigskip

\subheading{The construction of the approximation domain.}

\begin{settingnotation} \label{7.4.1} Let $(T,\n)$ be a
local Krull domain with  field of fractions
$F$. Assume there exists a nonzero element
$y\in \n$
such that the $y$-adic completion $\widehat{(T,(y))} := (T^*,\n^*)$
of $T$  is an  analytically normal Noetherian local domain.
It then follows that the $\n$-adic completion $\widehat
T$ of $T$ is also a normal Noetherian local domain, since the $\n$-adic
completion of $T$ is the same as the $\n^*$-adic completion
of $T^*$.  Since $T^*$ is Noetherian,
if $F^*$ denotes the field of fractions  of $T^*$, then
$T^* = \widehat T \cap F^*$. Therefore
$F\cap T^*=F\cap\widehat T$. Let $d$ denote the dimension of
the Noetherian domain $T^*$. It follows that $d$ is also
the dimension of $\widehat T$.
\footnote{If $T$ is Noetherian, then $d$ is also the dimension of $T$.  However,
if $T$ is  not Noetherian,  then the dimension of $T$ may be greater than  $d$.
This is illustrated by
taking $T$ to be the ring $B$ of Example~\ref{4.7.13}.}

(1) Assume that $T=F\cap T^*=F\cap\widehat T$, or equivalently by
Proposition~\ref{wfimph1p}.1,  that $T^*$ and $\widehat T$ are weakly flat over $T$.

(2) Let  $\widehat T[1/y]$ denote the
localization of $\widehat T$ at the powers of $y$, and
similarly, let  $T^*[1/y]$  denote the
localization of $T^*$ at the powers of $y$. The domains $\widehat T[1/y]$ and
$T^*[1/y]$ have dimension $d-1$.

(3) Let $\tau_1,\dots,\tau_s\in \n^*$ be algebraically
independent over $F$.

(4) For each $i$ with $1\le i\le s$, we have an
expansion  $\tau_i:=\Sigma_{j=1}^\infty c_{ij}y^j$ where
$c_{ij} \in T$.

(5) For each $n \in \N$ and each $i$ with  $1\le i\le s$, we define the
$n^{\text{th}}$-endpiece $\tau_{in}$ of $\tau_i$
with respect to $y$  as in Example~\ref{easy}:
\begin{equation*}\tau_{in}~:= ~ \Sigma_{j=n+1}^\infty c_{ij}y^{j-n}.\tag{\ref{7.4.1}.0}
\end{equation*}
Thus we have $\quad \tau_{in} =
y\tau_{i,n+1}+c_{i,n+1}y.$

(6) For each $n \in\N$, we define
$B_n:= T[\tau_{1n}, \dots,\tau_{sn}]_{(\n,\tau_{1n},\dots,\tau_{sn})}$.
In view of (5), we have $B_n \subseteq B_{n+1}$ and $B_{n+1}$ dominates
$B_n$ for each $n$. We define
$$  B ~:= ~\varinjlim_{n\in \N} B_n ~= ~\bigcup_{n=1}^\infty B_n, \quad
\text{ and } \quad
A ~:= ~F(\tau_1,\dots,\tau_s)~\cap ~\widehat T.$$
Thus, $B$ and $A$ are local Krull domains and $A$ birationally
dominates $B$. We are especially interested in conditions which imply
that $B=A$.

(7) Let $A^*$ denote the $y$-adic completion
 of $A$
and let $B^*$  denote the $y$-adic completion of $B$.
\end{settingnotation}

\medskip

\begin{remarks}  \label{7.4.3} The definitions of $B$ and
$B_n$ are
independent of representations for $\tau_1,\dots,\tau_s$
as power series in  $y$ with coefficients in $T$; see  \cite[Proposition~21.6]{powerbook}.
\end{remarks}

\subheading{Properties of the construction.}

The following theorem is proved in \cite[Theorem~21.7]{powerbook}.

\begin{theorem} \label{7.4.4} Assume the setting and notation of (\ref{7.4.1}).
Then the intermediate rings $B_n, \ B$ and $A$ have the following
properties:
\begin{enumerate}
\item $yA = yT^* \cap A$ and $yB = yA \cap B = yT^* \cap B$.
More generally, for every $t \in \N$, we have
$y^tA = y^tT^* \cap A$ and $y^tB = y^tA \cap B = y^tT^* \cap B$.

\item $T/y^tT = B/y^tB = A/y^tA = T^*/y^tT^*$, for each positive
integer $t$.

\item Every  ideal of
$T, B$ or $A$ that contains $y$ is finitely generated by elements
of $T$. In particular, the maximal ideal $\n$ of $T$ is
finitely generated, and the maximal ideals of $B$ and $A$ are
$\n B$ and $\n A$.

\item For every $n\in \N$:
$yB\cap B_n = (y, \tau_{1n},\hdots,\tau_{sn})B_n$, an ideal of
$B_n$ of height $s+1$.

\item Let  $P\in \Spec(A)$ be minimal over $yA$,
and let  $Q=P\cap B$  and  $W=P\cap T$. Then
$T_{W}\subseteq  B_{Q}=A_{P}$, and  all three localizations are DVRs.

\item For every $n\in \N$, $B[1/y]$ is a localization of $B_n$,
i.e., for each $n \in \N$, there exists a
multiplicatively closed subset $S_n$ of $B_n$
such that $B[1/y]=S_n^{-1}B_n$.

\item $B = B[1/y] \cap B_{\q_1}\cap\dots\cap B_{\q_r}$, where
$\q_1,\dots,\q_r$ are the prime ideals of $B$ minimal over $yB$.
\end{enumerate}
\end{theorem}

 The next theorem  from \cite[Theorem~21.8]{powerbook} is also useful in the sequel.
\begin{theorem}\label{7.4.5} With the setting and notation of
(\ref{7.4.1}),
the intermediate rings $A$ and $B$ have the following
properties:
\begin{enumerate}
\item $A$ and $B$ are  local Krull domains.
\item $B\subseteq A$, with $A$ dominating $B$.
\item $A^*=B^*=T^*$.
\item If $B$ is Noetherian, then $B = A$.
\end{enumerate}
Moreover, if  $T$ is a unique
factorization domain (UFD) and $y$ is a prime element of $T$, then $B$
is a UFD.
\end{theorem}

We use the following theorem \cite[Theorem~21.13]{powerbook}  to
establish the Noetherian property.

\begin{theorem} \label{7.5.45}
Assume the notation of  Setting~\ref{7.4.1}.
Thus $(T,\n)$ is a local Krull domain with field of fractions $F$,
and $y \in \n$ is
such that the $(y)$-adic completion $(T^*, \n^*)$ of $T$ is an analytically normal
Noetherian local domain and $T = T^* \cap F$.  For
elements $\tau_1, \ldots, \tau_s \in \n^*$ that are algebraically
independent over $T$, the following are equivalent:
\begin{enumerate}
\item The extension $T[\tau_1, \ldots, \tau_s] \hookrightarrow T^*[1/y]$ is flat.
\item  The elements
$\tau_1,\dots,\tau_s$ are primarily limit-intersecting in $y$ over $T$.
\item The intermediate rings $A$ and $B$ are equal and are Noetherian.
\item The constructed ring $B$ is Noetherian.
\end{enumerate}
Moreover, if these equivalent conditions hold, then the Krull domain
$T$ is Noetherian.
\end{theorem}

We consider the following  properties of an extension of Krull domains.

\begin{definitions}\label{wfetc}
 Let $S\hookrightarrow T$ be an
extension of Krull domains.
\begin{enumerate}
\item  \label{wf12} We say that the extension $S \hookrightarrow T$ is
{\it weakly flat}, or  that $T$ is
{\it weakly flat} over $S$, if every
height-one prime
ideal $P$ of $S$ with $PT\neq T$ satisfies
$PT \cap S = P$.
\item  \label{h1p12} We say that the extension $S \hookrightarrow T$ is
{\it height-one preserving}, or that $T$ is a
{\it height-one preserving} extension  of $S$,  if for every
height-one prime ideal $P$ of $ S$ with $PT \ne T$ there exists  a
height-one prime ideal $Q$ of $T$ with  $PT\subseteq Q$.
\item  \label{Lfd12}  For
 $d \in \N$, we say that $\varphi:S \hookrightarrow T$
{\it satisfies LF$_d$}  ({\it locally flat in height $d$}),
if, for each $P \in \Spec  T$ with $\hgt P
\le d$, the composite map $S \to T \to T_P$ is flat.
\end{enumerate}
\end{definitions}

The condition LF$_1$ is equivalent to weak flatness. If $\dim T \le  d$, then
the condition LF$_d$ is equivalent to flatness.
Proposition~\ref{wfimph1p} demonstrates the relevance of the weak flatness property
for  an extension of Krull domains.

\begin{proposition}\label{wfimph1p} $\phantom{i}$ \cite[Corollary~12.4]{powerbook}
Let $\varphi: S\hookrightarrow T$ be an
extension of
Krull domains and let $F$ denote the field of fractions of $S$.
\begin{enumerate}
\item  Assume that $PT\neq T$ for every height-one
prime ideal $P$ of $S$.
Then $S \hookrightarrow T$  is
  weakly flat $\iff \, S = F\cap T$.
\item If $S \hookrightarrow T$ is weakly flat, then
$\varphi$ is
height-one preserving and, moreover, for every height-one
prime ideal $P$ of $S$ with $PT\neq T$, there is a height-one
prime ideal $Q$ of $T$ with $Q\cap S=P$.
\end{enumerate}

\end{proposition}

Theorem~\ref{7.5.5}   states that weak flatness
of the map $\varphi$ of Equation~\ref{Krullset}.0 is equivalent to equality of the
intersection domain $A$ with its approximation domain $B$.

\begin{theorem} \label{7.5.5}   \cite[Theorem~21.14]{powerbook}
 Assume the notation of  Setting~\ref{7.4.1}.
Thus $(T,\n)$ is a local Krull domain with field of fractions $F$,
and $y \in \n$ is
such that the $(y)$-adic completion $(T^*, \n^*)$ of $T$ is an analytically normal
Noetherian local domain and $T = T^* \cap F$.  For
elements $\tau_1, \ldots, \tau_s \in \n^*$ that are algebraically
independent over $T$, the following are equivalent:
\begin{enumerate}
\item
The intersection domain $A$ is equal to its approximation domain $B$.
\item The map $\varphi: T[\tau_1, \ldots, \tau_s]  \longrightarrow T^*[1/y]$ is weakly flat.
\item The map  $B\longrightarrow T^*[1/y]$ is weakly flat.
\item The map  $B\longrightarrow T^*$ is weakly flat.
\end{enumerate}
\end{theorem}

\section{Primarily limit-intersecting elements} \label{section3}

In this section, we establish the existence of primarily limit-intersecting elements over countable excellent normal local domains.

We use Corollary~\ref{16.2.3p} and Lemma~\ref{6.3.7}  in the proof of  Theorem~\ref{exprimli}.

\smallskip

\begin{theorem}\label{13.2.0p} $\phantom{i}$ \cite[Theorem 11.3]{powerbook}
Let $(R,\m), (S,\n)$ and $(T,{\bf \ell})$
be Noetherian local rings, and assume there exist local homomorphisms:
$$
R \longrightarrow S\longrightarrow T,$$
such that
\begin{enumerate}
\item [(i)] $R \to T$ is flat  and $T/\m T$  is Cohen-Macaulay.
\item [(ii)] $R \to S$ is flat and $S/\m S$ is a regular local ring.
\end{enumerate}
Then the following statements are equivalent:
\begin{enumerate}
\item $S \to T$ is flat.
\item For each prime ideal $ {\bf w}$ of $T$,
we have  $\text{ht}(\w)\ge \text{ht}(\w\cap S)$.
\item For each prime ideal $\w$ of $T$ such that $\w$ is
minimal over $\n T$,  we have
$\hgt(\w)\ge\hgt(\n)$.
\end{enumerate}
\end{theorem}

Since flatness is a local property, the
following  corollary  is immediate.

\begin{corollary} \label{16.2.3p}   Let   $R$, $S$ and $T$  be  Noetherian rings, and assume
there exist ring homomorphisms $R \to S \to T$. If the map  $R \to T$ is flat with
Cohen-Macaulay fibers  and the map  $R \to S$ is flat with regular fibers, then
the following two statements are equivalent:
\begin{enumerate}
\item The map  $S \to T$  is flat,
\item For each prime ideal $P$ of $T$,
we have  $\hgt(P)\ge \hgt(P\cap S)$.
\end{enumerate}
\end{corollary}

To establish  the existence of primarily  limit-intersecting
elements, we use the following prime avoidance lemma;   see the
articles \cite{Bu}, \cite{SV}, \cite{WW}  and the book
\cite[Lemma~14.2]{LW} for other prime avoidance results  involving
countably infinitely many prime ideals.

\begin{lemma} \label{6.3.7}   \cite[Lemma~22.10]{powerbook}  Let $(T,\n)$ be a
Noetherian local domain that is complete in the $(y)$-adic topology,  where $y$ is a nonzero
element of $\n$.   Let $\mathcal U$ be a countable set of  prime ideals of $T$ such
that $y \not\in P$ for each $P \in \mathcal U$, and fix an arbitrary element
 $t \in \n \setminus \n^2$.      Then there exists
 an element $a\in y^2T$ such that $t-a  ~\not\in  ~\bigcup
\{P : P \in \mathcal U \}$.
 \end{lemma}

\begin{proof}    We may assume there are no inclusion relations among the $P \in\mathcal U$.   We enumerate the prime ideals in
$\mathcal U$ as   $\{P_i\}_{i=1}^\infty$.       We choose $b_2 \in T$
so that $t - b_2y \not\in P_1$ as follows:  (i) if $t \in P_1$, let
$b_2  = 1$. Since $y \not\in P_1$,  we have $t - y^2 \not\in P_1$.
(ii) if $t \not\in P_1$,  let $b_2$ be a nonzero element of $P_1$.
Then $t - b_2y^2 \not\in P_1$.   Assume by induction that we have
found $b_2, \ldots, b_n$ in $T$ such that
$$
t - cy^2 ~:= ~
t - b_2y^2 - \cdots - b_ny^n ~ \not\in ~  P_1 \cup \cdots \cup P_{n-1}.
$$
We choose $b_{n+1} \in T$ so that $t - cy^2   -  b_{n+1}y^{n+1} \not\in \bigcup_{i=1}^n P_i$
as follows:
(i)  if $t - cy^2 \in P_n$, let $b_{n+1}  ~\in  ~(\prod_{i=1}^{n-1}P_i ) ~\setminus  ~P_n$.
(ii) if $t - cy^2 \not\in P_n$,
let $b_{n+1}$ be any nonzero element in $\prod_{i=1}^n P_i$.
Hence in either case there exists $b_{n+1} \in T$
so that
$$
t - b_2y^2 - \cdots - b_{n+1}y^{n+1} ~ \not\in P_1 \cup \cdots \cup P_n.
$$
Since $T$ is complete in the $(y)$-adic topology, the Cauchy sequence
$$
\{b_2y^2 + \cdots + b_ny^n \}_{n=2}^\infty
$$
has a limit $a \in \n^2$.   Since $T$ is Noetherian and local, every ideal of $T$ is
closed in the $(y)$-adic topology. Hence,
  for each integer $n \ge 2$, we have
$$
t - a~ =  ~(t-b_2y^2 - \cdots - b_ny^n) ~-~ (b_{n+1}y^{n+1}   + \cdots ),
$$
where $t -b_2y^2 - \cdots - b_ny^n    \not\in P_{n-1}$ and $(b_{n+1}y^{n+1} + \cdots )
\in  P_{n-1}$.
We conclude that  $t-a \not\in \bigcup_{i=1}^\infty P_i$.
\end{proof}

\subheading{The existence of one primarily limit-intersecting element.}

We use the following setting to describe necessary and sufficient conditions
for an element to be primarily limit-intersecting.

\begin{setting} \label{excset} Let $(R,\m)$  be a $d$-dimensional  excellent
normal local domain with $d \ge 2$,
let $y$ be a nonzero element of $\m$ and let $R^*$ denote the $(y)$-adic completion
of $R$.
Let $t$ be a variable over $R$,  let
$S:=R[t]_{(\m,t)}$, and let $S^*$ denote the $I$-adic completion of $S$,
where $I := (y,t)S$.  Then $S^* = R^*[[t]]$ is a $(d+1)$-dimensional normal Noetherian local
domain with maximal ideal $\n^* := (\m, t)S^*$. For each element $a \in y^2S^*$, we
have $S^* = R^*[[t]] = R^*[[t-a]]$ since $R^*$ is complete in the $(a)$-adic
topology. Let $\lambda_a: S^* \to R^*$ denote
the canonical homomorphism $S^* \to S^*/(t-a)S^* = R^*$, and let $\tau_a = \lambda_a(t)
= \lambda_a(a)$.
Consider the set
$$
\mathcal U ~:= ~\{P^* \in \Spec S^*~|~\hgt(P^*\cap S)= \hgt P^*,\text{ and } y\notin P^* ~\}.
$$
Since $S \hookrightarrow S^*$ is flat and thus satisfies the
going-down property, the set $\mathcal U$ can also be described as
the set of all $P^* \in \Spec S^*$ such that $y \notin P^*$ and
$P^*$ is minimal over $PS^*$ for some $P \in \Spec S$; see
\cite[Theorem~15.1]{M}
\end{setting}

\begin{theorem}  \label{iff} With the notation of Setting~\ref{excset},
the element $\tau_a$ is primarily limit-intersecting in $y$ over $R$ if and
only if $t-a \notin \bigcup\{P^* ~|~ P^* \in \mathcal U \}$.
\end{theorem}

\begin{proof}  Consider the commutative diagram:
$$\CD
\phantom{R} @.  S=R[t]_{(\m,t)}
@>{\subseteq}>> {S}^* = R^*[[t]] @>{\subseteq}>>S^*[1/y]\\
@. @V{\lambda_0}VV  @V{\lambda_a}VV @. {}\\
R @>{\subseteq}>> R_1=R[\tau_a]_{(\m,\tau_a)}
@>>> R^* @>{\subseteq}>>R^*[1/y].\endCD $$
\medskip

\centerline{Diagram~\ref{iff}.0}

\medskip
\noindent
The map  $\lambda_0$ denotes the restriction of $\lambda_a$ to $S$.

Assume that $\tau_a$ is primarily limit-intersecting in $y$ over $R$. Then $\tau_a$ is
algebraically independent over $R$ and $\lambda_0$ is an isomorphism.
If $t-a \in P^*$ for some $P^* \in \mathcal U$, we prove that $\varphi: R_1 \to R^*[1/y]$
is not flat.  Let $Q^* := \lambda_a(P^*)$.  We have $\hgt Q^* = \hgt P^* - 1$, and
$y \notin P^*$ implies $y \notin Q^*$. Let $P := P^* \cap S$ and $Q := Q^* \cap R_1$.
Commutativity of Diagram~\ref{iff}.0   and $\lambda_0$ an isomorphism imply
that $\hgt P = \hgt Q$.
Since $P^* \in \mathcal U$, we have $\hgt P = \hgt P^*$. It follows that
$\hgt Q > \hgt Q^*$. This implies that $\varphi: R_1 \to R^*[1/y]$ is not flat.

For the converse, assume that   $t-a\notin \bigcup\{ P^*~|~ P^*\in\mathcal U\}$.
Since $a \in y^2S^*$ and $S^*$ is complete in the $(y,t)$-adic topology, we have
$S^* = R^*[[t]] = R^*[[t-a]]$. Thus
$$
\mathfrak p~:= ~\ker(\lambda_a)~=~ (t- \tau_a){S}^* ~=~ (t-a)S^*
$$
is a height-one prime ideal of $S^*$.  Since $y \in R$ and $\mathfrak p \cap R = (0)$,
we have $y \notin \mathfrak p$.

Since $t-a$ is outside every element of $\mathcal U$, we have $\mathfrak p\notin \mathcal U$.
Since $\mathfrak p$ does not fit
the condition of $\mathcal U$, we have
$\hgt(\mathfrak p\cap S)\ne \hgt\mathfrak p=1$, and so, by the
faithful flatness of $S\hookrightarrow S^*$,
$\mathfrak p \cap S = (0)$.  Therefore the map
$\lambda_0 : S \to R_1$ has trivial kernel, and so $\lambda_0$
is an isomorphism.  Thus $\tau_a$ is algebraically independent over $R$.

Since $R$ is excellent and $R_1$ is a localized polynomial ring over $R$, the
hypotheses of Corollary~\ref{16.2.3p} are satisfied for the extension $R\hookrightarrow R_1\hookrightarrow R^*[1/y]$. It follows
that  the element  $\tau_a$ is primarily
limit-intersecting in $y$ over $R$
if $\hgt(Q_1^*\cap R_1)\le\hgt Q_1^*$
 for every prime ideal $Q_1^*\in \Spec(R^*[1/y])$,  or, equivalently, if, for every
 $Q^* \in \Spec R^*$  with $y \notin Q^*$, we  have
$\hgt(Q^* \cap R_1) \le \hgt Q^*$.
Thus, to
complete the proof of Theorem~\ref{iff}, it suffices to prove Claim~\ref{htler}.

\medskip

\begin{claim} \label{htler} For every prime
 ideal $Q^*\in \Spec R^*$ with $y\notin
Q^*$, we have $$\hgt(Q^*\cap R_1)\le \hgt Q^*.$$
\end{claim}
\noindent {\it Proof of Claim~\ref{htler}.} Since $\dim R^*=d$ and $y\notin Q^*$,
we have $\hgt Q^* =r\le d-1$. Since  the map $R\hookrightarrow R^*$ is flat,
we have $\hgt (Q^*\cap R)\le \hgt
Q^*  =  r$.
Suppose that  $Q:=Q^*\cap R_1$ has height at least $r+1$ in $\Spec R_1$.
Since $R_1$ is a localized polynomial ring in one variable over $R$ and
$\hgt(Q\cap R)\le r$, we have
$\hgt(Q) =  r+1$.
Let  $P:=\lambda_0^{-1}(Q)\in\Spec S$.
Then $\hgt P = r+1$ and $y\notin P$.

Let $P^* := \lambda_a^{-1} (Q^*)$. Since the prime ideals of $S^*$ that
contain $t-a$ and have height $r+1$ are in one-to-one correspondence with the prime ideals
of $R^*$ of height $r$, we have $\hgt P^* = r+1$. By the commutativity of the diagram,
we also have  $y\notin P^*$ and $P\subseteq P^*\cap S$, and so
$$r+1=\hgt P\le \hgt(P^*\cap S)\le \hgt P^* = r+1,$$ where the last inequality
holds because  the map
$S \hookrightarrow S^*$
is flat. It follows that $P=P^*\cap S$, and so $P^*\in\mathcal U$. This contradicts the fact  that $t-a \notin P_1^*$ for each
 $P_1^*\in\mathcal U$.  Thus we have $\hgt(Q^*\cap R_1)\le r=\hgt Q^*$,
as asserted in  Claim~\ref{htler}.

This completes the proof of Theorem~\ref{iff}.
\end{proof}

Theorem~\ref{iff} yields a necessary and sufficient condition for
an element of $R^*$ that is algebraically independent over $R$ to be
primarily limit-intersecting in $y$ over $R$.

\begin{remarks}  \label{iffremark}
Assume notation as in Setting~\ref{excset}.
\begin{enumerate}
\item For each $a \in y^2S^*$ as in Setting~\ref{excset}, we have $(t-a)S^* =
(t - \tau_a)S^*$. Hence $t-a \notin \bigcup\{P^* ~|~ P^* \in \mathcal U \} ~
\iff ~ t-\tau_a \notin \bigcup\{P^* ~|~ P^* \in \mathcal U \}$.

\item If $a \in R^*$, then the commutativity of Diagram~\ref{iff}.0 implies that $\tau_a = a$.

\item   For $\tau \in R^*$, we have
$\tau = b_0 + b_1y + \tau'$, where $b_0$ and $b_1$ are in $R$ and $\tau' \in y^2R^*$.

\begin{enumerate}
\item  The rings  $R[\tau]$ and $R[\tau']$ are equal.  Hence
$\tau$ is primarily limit-intersecting in $y$ over $R$
if and only if $\tau'$ is  primarily limit-intersecting in $y$ over $R$.

\item  Assume  $\tau \in R^*$ is algebraically independent over $R$. Then
$\tau$ is primarily limit-intersecting in $y$ over $R$ if and only if
$t-\tau' \notin \bigcup\{P^* ~|~ P^* \in \mathcal U \}$.
\end{enumerate}
Item 3b follows from Theorem~\ref{iff} by setting $a = \tau'$ and applying
item~3a and item~2.

\end{enumerate}

\end{remarks}

We use Theorem~\ref{iff} and Lemma~\ref{6.3.7} to prove
Theorem~\ref{exprimli}.

\begin{theorem}\label{exprimli} Let $(R, \m)$ be a countable excellent normal  local domain  of
dimension $ d \ge 2$,  let  $y$ be  a nonzero element in $\m$, and let    $R^*$ denote the
$(y)$-adic completion of $R$.   Then there exists
an  element $\tau  \in yR^*$ that  is primarily limit-intersecting in $y$ over $R$.
\end{theorem}
\begin{proof} As in Setting~\ref{excset}, let
$$
\mathcal U ~:= ~\{P^*\in \Spec S^*~|~\hgt(P^*\cap S)= \hgt P^*,~\text{ and } y\notin P^* ~\}.
$$
Since the ring $S$ is countable and Noetherian, the set $\mathcal U$ is countable.
 Lemma~\ref{6.3.7} implies that   there exists
an element $a\in y^2{S}^*$ such that $t-a\notin \bigcup\{ P^*~|~ P^*\in\mathcal U\}$.
By Theorem~\ref{iff}, the element $\tau_a$ is
primarily limit-intersecting in $y$ over $R$.
\end{proof}

\subheading{The existence of more  primarily limit-intersecting elements.}

To establish the existence of more than one primarily limit-intersecting element
we use the following setting.

\begin{setting} \label{excset2} Let $(R,\m)$  be a $d$-dimensional  excellent
normal local domain,
let $y$ be a nonzero element of $\m$ and let $R^*$ denote the $(y)$-adic completion
of $R$. Let  $t_1, \ldots, t_{n+1}$ be
indeterminates over $R$, and let $S_{n}$ and $S_{n+1}$ denote the localized polynomial rings
$$
S_n~ := ~ R[t_1,\dots,t_n]_{(\m,t_1,\dots,t_n)} \quad  \text{and} \quad
S_{n+1}~ := ~ R[t_1,\dots,t_{n+1}]_{(\m,t_1,\dots,t_{n+1})}.
$$
 Let $S_n^*$ denote the $I_n$-adic
completion of $S_n$, where $I_n := (y, t_1, \ldots, t_n)S_n$. Then $S_n^* =
R^*[[t_1, \ldots, t_n]]$ is a $(d+n)$-dimensional normal Noetherian local domain
with maximal ideal $\n^* = (\m, t_1, \ldots, t_n)S_n^*$.
Assume that $\tau_1, \ldots, \tau_n \in yR^*$
are primarily limit-intersecting  in $y$ over $R$, and define
$\lambda: S_n^* \to R^*$ to be the $R^*$-algebra
homomorphism such that $\lambda(t_i) = \tau_i$, for $1 \le i \le n$.

Since $S_n^* = R^*[[t_1 - \tau_1, \ldots, t_n - \tau_n]]$,  we have
$\p_n :=  \ker \lambda = (t_1 - \tau_1, \ldots, t_n - \tau_n)S_n^*$.
Consider  the commutative diagram:
$$\CD
\phantom{R} @.  S_n=R[t_1, \ldots, t_n]_{(\m,t_1, \ldots, t_n)}
@>{\subseteq}>> S_n^* = R^*[[t_1, \ldots, t_n]] @>{\subseteq}>>S_n^*[1/y]\\
@. @V{\lambda_0, ~\cong}VV  @V{\lambda}VV @. {}\\
R @>{\subseteq}>> R_n=R[\tau_1, \ldots, \tau_n]_{(\m,\tau_1, \ldots, \tau_n)}
@>{\varphi_0}>> R^* @>{\alpha}>>R^*[1/y].\endCD
$$
\vskip 3pt

\noindent
Let $S_{n+1}^*$
denote the $I_{n+1}$-adic completion of $S_{n+1}$, where
$I_{n+1} := (y, t_1, \ldots, t_{n+1})S_{n+1}$.
For each element $a \in y^2S_{n+1}^*$, we
have 
\begin{equation*} 
S_{n+1}^* = S_n^*[[t_{n+1}]] = S_n^*[[t_{n+1}-a]].  \tag{\ref{excset2}.1}
\end{equation*} 
Let $\lambda_a: S_{n+1}^* \to R^*$ denote
the composition
$$ \CD
S_{n+1}^* ~=~ S_n^*[[t_{n+1}]] ~@>>> ~\frac{S_n^*[[t_{n+1}]]}{(t_{n+1} - a)}
~=~ S_n^*  @>{\lambda}>> R^*,
\endCD
$$
and let $\tau_a := \lambda_a(t_{n+1}) = \lambda_a(a)$. We have
$\ker \lambda_a = (\p_n, t_{n+1} - a)S_{n+1}^*$.  Consider the commutative diagram
$$\CD
\phantom{R} @.  S_n
@>{\subseteq}>> S_n^*  @>{\subseteq}>>S_{n+1}^*  @>>> S_{n+1}^*[1/y] \\
@. @V{\lambda_0, ~\cong}VV  @V{\lambda}VV @V {\lambda_a}VV @VVV\\
R @>{\subseteq}>> R_n
@>{\varphi_0}>> R^* @>{=}>>R^* @>>> R^*[1/y].\endCD
$$
\vskip 5 pt
\centerline{Diagram~\ref{excset2}.2}
\vskip 1 pt

\noindent
Let
$$
\mathcal U ~:= ~\{P^*\in \Spec S_{n+1}^*~|~ P^* \cap S_{n+1} = P, ~y\notin P ~\text{ and }
P^*\text{ is minimal over } (P,  \p_n)S_{n+1}^*  \}.
$$

\noindent
Notice that  $y \notin P^*$ for each $P^* \in \mathcal U$, since $y \in R$ implies $\lambda_a(y)=y$.
\end{setting}

\begin{theorem} \label{iff2} With the notation of Setting~\ref{excset2}, the elements
$\tau_1, \ldots, \tau_n, \tau_a$ are primarily limit-intersecting in $y$ over $R$ if
and only if $t_{n+1} - a \notin \bigcup\{P^* ~|~ P^* \in \mathcal U \}$.
\end{theorem}
\begin{proof}
Assume that $\tau_1, \ldots, \tau_n, \tau_a$  are primarily limit-intersecting in $y$
over $R$. Then $\tau_1, \ldots, \tau_n, \tau_a$ are algebraically independent over $R$.
Consider the following commutative diagram:
$$\CD
\phantom{R} @.  S_{n+1}=R[t_1, \ldots, t_{n+1}]_{(\m,t_1, \ldots, t_{n+1})}
@>{\subseteq}>> S_{n+1}^* = R^*[[t_1, \ldots, t_{n+1}]] \\
@. @V{\lambda_1} VV  @V{\lambda_a}VV \\
R @>{\subseteq}>> R_{n+1}=R[\tau_1, \ldots, \tau_a]_{(\m,\tau_1, \ldots, \tau_a)}
@>>> R^*.\endCD $$
\medskip

\centerline{Diagram~\ref{iff2}.0}
\vskip 3 pt
\noindent The map $\lambda_1$ is the restriction of $\lambda_a$ to $S_{n+1}$,
and is an isomorphism since $\tau_1, \ldots, \tau_n, \tau_a$ are
algebraically independent over $R$.

If $t_{n+1}-a \in P^*$ for some $P^* \in \mathcal U$,
we prove that $\varphi: R_{n+1} \to R^*[1/y]$
is not flat, a contradiction to our assumption that $\tau_1, \ldots, \tau_n, \tau_a$
are primarily limit-intersecting.  Since $P^* \in \mathcal U$, we have $\p_n \subset P^*$.
Then
$t_{n+1}-a \in P^*$ implies $\ker \lambda_a \subset P^*$. Let $\lambda_a(P^*) := Q^*$.
Then $\lambda_a^{-1}(Q^*) = P^*$ and $\hgt P^* = n+1 + \hgt Q^*$.
Since $P^* \in \mathcal U$, we have $y \notin P^*$. The
commutativity of Diagram~\ref{iff2}.0  implies that $y \notin Q^*$.
  Let $P := P^* \cap S_{n+1}$ and let $Q := Q^* \cap R_{n+1}$.
Commutativity of Diagram~\ref{iff2}.0   and $\lambda_0$ an isomorphism imply
that $\hgt P = \hgt Q$.  Since $P^*$ is a minimal prime of $(P, \p_n)S_{n+1}^*$,
$\p_n$ is $n$-generated, and $S_{n+1}^*$ is Noetherian and catenary,
we have $\hgt P^* \le \hgt P + n$. Hence $\hgt P \ge \hgt P^* - n$. Thus
$$
\hgt Q ~=~ \hgt P  ~\ge ~ \hgt P^* - n ~=~ \hgt Q^* + n+1 - n ~=~ \hgt Q^* + 1.
$$
The fact that
$\hgt Q > \hgt Q^*$ implies that the map $R_{n+1} \to R^*[1/y]$ is not flat.
This proves the forward direction.

For the converse,  we have  
\smallskip

\noindent{\bf Assumption~\ref{iff2}.1:} \quad $t_{n+1} - a~ \notin ~ \bigcup ~ \{~P^* ~|~ P^* \in \mathcal U ~\}.$

\smallskip

\noindent
Since $\lambda_a: S_{n+1}^* \to R^*$ is an extension of $\lambda:S_n^* \to R^*$ as in
Diagram~\ref{excset2}.2, we have $\ker \lambda_a \cap S_n = (0)$.
Let $\p := (t_{n+1} - a)S_{n+1}^*= (t_{n+1} - \tau_a)S_{n+1}^*$.
As in Equation~\ref{excset2}.1, we have
$$ 
S^*_{n+1} ~ = ~ R^*[[t_1, \ldots, t_{n+1}]] ~ = ~ R^*[[t_1 - \tau_1, \ldots, t_{n} - \tau_n, t_{n+1} - a]].
$$
Thus $P^* := (\p_n, \p)S^*_{n+1}$ is a prime ideal of height $n+1$ and 
$P^* \cap R^* = (0)$. It follows that $y \notin P^*$.  
We show that $P^* \cap S_{n+1} = (0)$.  Assume that 
$P = P^* \cap S_{n+1} \ne (0).$  Since $\hgt P* = n+1$, $P^*$ is
minimal over $(P, \p_n)S_{n+1}^*$, and so $P^* \in \mathcal U$, a contradiction 
to Assumption~\ref{iff2}.1. Therefore $P^* \cap S_{n+1} = (0)$.
It follows that 
$\p  \cap S_{n+1} = (0)$  since $\p \subset P^*$. 
Thus $\ker\lambda_1=(0)$, and so $\lambda_1$ in Diagram~\ref{iff2}.0 is an isomorphism.
Therefore $\tau_a$ is algebraically independent
over $R_n$.

Since $R$ is excellent and $R_{n+1}$ is a localized polynomial ring
in $n+1$ variables over $R$, the
hypotheses of Corollary~\ref{16.2.3p} are satisfied for the composition
$$R\hookrightarrow R_{n+1}\hookrightarrow R^*[1/y].$$
 It follows
that  the elements  $ \tau_1, \ldots, \tau_n, \tau_a$  are  primarily
limit-intersecting in $y$ over $R$    if, for every
 $Q^* \in \Spec R^*$  with $y \notin Q^*$, we  have
$\hgt(Q^* \cap R_{n+1}) \le \hgt Q^*$.
Thus, to
complete the proof of Theorem~\ref{iff2}  with $\tau_{n+1}=\tau_a$, it suffices to prove Claim~\ref{htler2}.

\medskip

\begin{claim} \label{htler2}
 Let  $Q^*\in \Spec R^*$ with $y \notin Q^*$ and $\hgt Q^*=r$.
Then  $$\hgt(Q^*\cap R_{n+1})\le~r.$$
\end{claim}

\noindent{\it Proof of Claim~\ref{htler2}.} Let $Q_1:=Q^*\cap R_{n+1}$ and let  $Q_0:=Q^*\cap R_{n}$.
 Suppose   $\hgt Q_1>r$.
Notice that $r<d$, since $d=\dim R^*$ and $y\notin Q^*$.

Since  $\tau_1,\dots, \tau_{n}$ are primarily limit-intersecting in $y$ over $R$, the extension
$$R_{n}:=R[\tau_1,\dots, \tau_{n}]_{(\m,\tau_1,\dots, \tau_{n})}\hookrightarrow  R^*[1/y]$$
from Diagram~\ref{excset2}.2
is flat. Thus $\hgt Q_0\le r$ and $ \hgt Q_0\le\hgt L^*$ for every prime ideal $L^*$ of $R^*$ with $Q_0R^*\subseteq L^*\subseteq Q^*$.
 Since $R_{n+1}$ is a localized
polynomial ring in the indeterminate $\tau_a$ over $R_n$, we have that  $\hgt Q_1\le
\hgt Q_0+1=r+1$. Thus  $\hgt Q_1=r+1$  and $\hgt Q_0=r$.
It follows that  $Q^*$ is a minimal prime of $Q_0R^*$.

Let $h(\tau_{a})\in R_n[\tau_a]=R_{n+1}$ be a  polynomial in the variable $\tau_a$ over the ring $R_n$ such that
$$
h(\tau_{a})\in (Q^*\cap R_n[\tau_{a}]) ~ \setminus  ~ (Q^*\cap R_n)R_{n+1}.
$$
It follows that
$Q_1$ is  a minimal prime of the ideal  $( Q_0,h(\tau_{a}))R_{n+1}$.

With notation from Diagram~\ref{excset2}.2, define
$$P_0~:=~\lambda_{0}^{-1}( Q_0)~\text{ and }~P_0^*~ := \lambda^{-1}(Q^*).$$
Since $\lambda_0$ is an isomorphism,  $P_0$ is a prime ideal of $S_n$ with  $\hgt P_0=r$.
Moreover,  we have the following:
\begin{enumerate}
\item $P_0^* \cap S_n = P_0$ ~(by commutativity in Diagram~\ref{excset2}.2),
\item $y \notin P_0^*$~ (by item 1),
\item $P_0^*$ is a minimal prime of $(P_0, \p_n)S_n^*$~ (since $S_n^*/\p_n = R^*$ in Diagram~\ref{excset2}.2, and $Q^*$
is a minimal prime of $Q_0R^*$),
\item $\hgt P_0^* = n + r$ ~(by the correspondence between prime ideals of $S_n^*$ containing $\p_n$ and prime ideals of $R^*$).  \end{enumerate}

Consider  the  commutative diagram below with the left and right ends identified:
\begin{equation*}\CD  S_{n+1}^*@<<< S_n^*@<<< S_n  @>>>  S_{n+1} @>{\theta}>> S_{n+1}^* \\
@V{\lambda_a}VV @V{\lambda}VV @V{\lambda_{0}, \cong}VV  @V{\lambda_1, \cong}VV  @V{\lambda_a}VV \\
R^* @<<< R^* @<<<  R_n @>>> R_{n+1}@>>> R^*,\endCD
\end{equation*}
\centerline{Diagram~\ref{htler2}.0}

 \vskip 3 pt
 \noindent
where $\lambda, \lambda_0$ and $\lambda_1$ are as in Diagrams~\ref{excset2}.2 and \ref{iff2}.0,
and so $\lambda_a$ restricted to $S_n^*$ is $\lambda$.
Let $h(t_{n+1})~=~\lambda_1^{-1}(h(\tau_a)) $ and set
$$
P_1~  :=  ~ \lambda_1^{-1}(Q_1) ~ \in  ~\Spec(S_{n+1}),\quad ~\text{ and }~P^* ~:= ~\lambda_a^{-1}(Q^*) \in \Spec(S_{n+1}^*).
$$
 Then $P_1$ is a minimal prime of $(P_0, h(t_{n+1}))S_{n+1}$,  since  $Q_1$ is a minimal prime of
$(Q_0, h(\tau_a))R_{n+1}$.
Since $Q_1\subseteq Q^*$,  we have $h(t_{n+1})\in P^*$ and $P_1S_{n+1}^*\subseteq  P^*$ because ~ $\lambda_a(h(t_{n+1}))=\lambda_1(h(t_{n+1}))=h(\tau_a)\in Q_1$ and $\lambda_a(P_1)=\lambda_1(P_1)=Q_1$.
 By the correspondence between prime ideals of
$S_{n+1}^*$ containing $\ker(\lambda_a)=\p_{n+1}$ and prime ideals of $R^*$, we see
$$\hgt P^*=\hgt Q^*+n+1=r+ n+1.$$
Since $\lambda_a(P_0^*)\subseteq Q^*$, we have $P_0^*\subseteq P^*$, but $h(t_{n+1})\notin P_0$ implies $h(t_{n+1})\notin P_0^*S_{n+1}^*$. Therefore  $$(P_0,\p_n)S_{n+1}^* \subseteq P_0^*S_{n+1}^*\subsetneq (P_0^*,h(t_{n+1}))S_{n+1}^*\subseteq P^*.$$
By items~3 and 4 above, $\hgt P_0^*=n+r$ and $P_0^*$ is a minimal prime of $(P_0, \p_n)S_n^*$. Since $\hgt P^*=n+r+1$, it follows that $P^*$ is a minimal prime of $(P_0, h(t_{n+1}),\p_n)S_{n+1}^*$. Since $(P_0, h(t_{n+1}),\p_n)S_{n+1}^*\subseteq (P_1,\p_n)S_{n+1}^*\subseteq P^*$, we have $P^*$ is a minimal prime of $(P_1,\p_n)S_{n+1}^*$.
But then, by Assumption~\ref{iff2}.1 on $\mathcal U$, we have $t_{n+1}-a\notin P^*$,
a contradiction.
This contradiction implies  that $\hgt Q_1=r$.

This completes the proof of  Claim~\ref{htler2} and thus
 the proof of Theorem~\ref{iff2}.\end{proof}

We use Theorem~\ref{exprimli},  Theorem~\ref{iff2} and Lemma~\ref{6.3.7} to prove in
Theorem~\ref{exnprimli}  the existence
over a countable excellent normal local
domain of dimension at least two
of an infinite sequence
of primarily limit-intersecting elements.

\begin{theorem} \label{exnprimli}  Let $R$ be a countable excellent normal local domain  of
dimension $d \ge 2$, let $y$ be a nonzero element in
the maximal ideal $\m$ of $R$, and let $R^*$ be the $(y)$-adic completion of $R$.
Let $n$ be a positive integer.  Then
\begin{enumerate}
\item If the elements $\tau_1, \ldots, \tau_{n} \in yR^*$ are primarily
limit-intersecting in $y$ over $R$, then there exists an element $\tau_a \in yR^*$
such that $\tau_1, \ldots, \tau_{n}, \tau_a$ are primarily limit-intersecting in $y$ over $R$.
\item  There exists an infinite sequence $\tau_1, \ldots, \tau_n, \ldots \in yR^*$
of elements that are primarily limit-intersecting in $y$ over $R$.
\end{enumerate}
\end{theorem}

\begin{proof}
By Definition~\ref{prlidef}, item 1 implies item 2; thus it suffices to prove item~1.
Theorem~\ref{exprimli}  implies the existence
of an element $\tau_1 \in yR^*$ that is primarily limit-intersecting in $y$ over $R$.
 As in Setting~\ref{excset2}, let
$$
\mathcal U ~:= ~\{P^*\in \Spec S_{n+1}^*~|~ P^* \cap S_{n+1} = P, ~y\notin P ~\text{ and }
P^*\text{ is minimal over } (P,  \p_n)S_{n+1}^*  \}.
$$
Since the ring $S_{n+1}$ is countable and Noetherian, the set $\mathcal U$ is countable.
Lemma~\ref{6.3.7} implies that there exists an element $a \in y^2S_{n+1}^*$ such that
$$
t_{n+1} - a  ~\notin  ~  \bigcup \{P^* ~|~ P^* \in \mathcal U \}.
$$
By Theorem~\ref{iff2}, the elements $\tau_1, \ldots, \tau_n, \tau_a$
are primarily limit-intersecting in $y$ over $R$.
\end{proof}

\subheading{Normal Noetherian domains  that are not excellent.}

Using Theorem~ \ref{exprimli}, we establish in Theorem~\ref{nagtheorem},
for every countable  excellent normal local domain $R$  of
dimension $d \ge 2$,   the existence of a
primarily limit-intersecting  element $\eta \in yR^*$ such
that the   constructed Noetherian
domain
$$
B ~= ~ A ~ =  ~ R^* \cap \mathcal Q(R[\eta])
$$ is not a Nagata domain and hence is not excellent.\footnote{``Nagata" is defined after Question~\ref{RamQ}; see also  \cite[Definitions~2.3.1, 3.28]{powerbook}, \cite[pages 264, 260]{M}.}

\begin{theorem} \label{nagtheorem}
Let $R$ be a countable  excellent normal local domain  of
dimension $d \ge 2$, let $y$ be a nonzero element in
the maximal ideal $\m$ of $R$, and let $R^*$ be the $(y)$-adic completion of $R$.
There exists an element $\eta \in yR^*$ such that
\begin{enumerate}
\item  $\eta$  is primarily limit-intersecting in $y$  over $R$.
\item The associated intersection  domain $A := R^* \cap \mathcal Q(R[\eta])$  is equal
to its approximation domain $B$.
\item The ring $A$  has  a height-one prime ideal $\p$ such
that $R^*/\p R^*$ is not reduced.
\end{enumerate}
Thus the integral  domain $A = B$
associated to $\eta$ is a  normal Noetherian local domain that is not a Nagata
domain and hence is not excellent.
\end{theorem}

\begin{proof}
Since $\dim R \ge 2$, there exists $x \in \m$ such that $\hgt (x,y)R = 2$.
By Theorem~ \ref{exprimli},  there
exists $\tau \in yR^*$ such that $\tau$ is primarily limit-intersecting in $y$ over $R$.
Hence the extension
$R[\tau] \longrightarrow R^*[1/y]$ is flat.   Let $n \in \N$ with $n \ge 2$, and let
$\eta := (x + \tau)^n$.
Since $\tau$ is algebraically independent over $R$,
the element $\eta$ is also algebraically independent over $R$.  Moreover,
the polynomial ring $R[\tau]$ is a free
$R[\eta]$-module with $1, \tau, \ldots, \tau^{n-1}$ as a free module basis.  Hence the
map $R[\eta] \longrightarrow R^*[1/y]$ is flat.   It follows
that $\eta$ is primarily limit-intersecting in $y$ over $R$.  Therefore
the intersection domain $A := R^* \cap \mathcal Q(R[\eta])$
is equal to its associated approximation domain $B$  and is a normal
Noetherian domain with $(y)$-adic completion $R^*$.
Since $\eta$ is a prime element of the polynomial ring $R[\eta]$ and $B[1/y]$
is a localization of $R[\eta]$, it follows that
$\p := \eta B$
is a  height-one prime ideal of $B$.   Since $\tau \in R^*$, and $\eta = (x + \tau)^n$,
the ring $R^*/\p R^*$ contains nonzero
nilpotent elements.  Since $\widehat R = \widehat B$ is faithfully flat over $R^*$,
it follows that $\widehat B/\p\widehat B$ has nonzero nilpotent elements.
Since a Nagata local domain is analytically unramified, it follows
 that the normal Noetherian domain $B$ is not a Nagata ring,  \cite[page~264]{M} or
\cite[(32.2)]{N2}.
\end{proof}

\section{Other results and examples using the construction} \label{section4}

We use the following notation for the beginning of this section,  and make several remarks
concerning properties of and relationships among the integral domains being considered.

\begin{notationremarks}\label{4.7.n}
Let $k$ be a field, let $x$ and $y$ be
indeterminates over $k$,  and let
$$
\sigma ~ := ~ \sum_{i=1}^\infty a_ix^i ~ \in ~ xk[[x]] \qquad \ {\text{\rm and
}} \qquad \
\tau ~:= ~\sum_{i=1}^\infty b_iy^i ~\in ~ yk[[y]]
$$
be formal power series that are algebraically independent over the
fields $k(x)$ and $k(y)$, respectively. Let $R:=k[x,y]_{(x,y)}$, and
let $\sigma_n$, $\tau_n$ be the $n^{\text{th}}$ endpieces of
$\sigma,\tau$  respectively.   Define

\begin{equation}\aligned C_n &:=~ k[x,\sigma_n]_{(x,\sigma_n )},
 \quad C := k(x,\sigma) \cap k[[x]]= \varinjlim(C_n) = \bigcup_{n=1}^\infty C_n \ ; \\
D_n& :=~ k[y,\tau_n]_{(y,\tau_n)},\quad D := k(y,\tau)\cap k[[y]]  =
 \varinjlim(D_n) = \bigcup_{n=1}^\infty D_n;
  \\
  U_{n} &:= ~k[x,y,\sigma_n,\tau_n],   \quad U := \varinjlim U_{n} = \bigcup_{n=1}^{\infty} U_n; \\
   B_{n} &:= ~{k[x,y,\sigma_n,\tau_n]}_{(x,y,\sigma_n , \tau_n )}\quad B :=
   \varinjlim (B_{n}) = \bigcup_{n=1}^{\infty} B_n;\\
   A & := k(x,y,\sigma,\tau )\,
\cap \, k[[x,y]].
\endaligned \tag{\ref{4.7.n}.0}\end{equation}

\vskip 3pt

 Since  $k[[x,y]]$ is the $(x,y)$-adic completion of
the Noetherian ring $R$,  the ring
$k[[x,y]] = \widehat R$ is faithfully flat over $R$. Hence we have
$$
(x,y)^nk[[x,y]] ~\cap ~ R ~= ~(x,y)^nR
$$
for each $n \in \N$.

The relationships
$$
\sigma_{n} ~= ~ -xa_{n+1}+ x\sigma_{n+1} \quad \text{and} \quad  \tau_n=-yb_{n+1}+ y\tau_{n+1}
$$
among the endpieces imply for each positive integer $n$ the
inclusions
$$
C_n ~\subset ~ C_{n+1}, \quad D_n ~ \subset ~D_{n+1},\quad \text{ and }
\quad B_{n}~ \subset ~B_{n+1}.
$$
Moreover, for each of these inclusions we have  birational
domination of the larger local ring over the smaller,  and the local rings
$C_n, D_n, B_n$ are all dominated by $k[[x,y]] = \widehat R$.

Since $(x, y,\sigma_n, \tau_n)U_n$ is a maximal ideal of $U_n$ that is contained in
$(x,y)U$, a proper ideal of $U$, it follows that $(x,y)U \cap U_n= (x,
y, \sigma_n, \tau_n)U_n$. Since $B_n$ is the localization of the polynomial
ring $U_n$ at the maximal ideal $(x, y,\sigma_n, \tau_n)U_n$, we have
 $(x,y)B\cap
B_n=(x,y,\sigma_n,\tau_n)B$ for each $n \in \N$.

We have
 $\sigma_{n+1}\in U_n[\frac 1x]\subseteq U_n[\frac
1{xy}]$ and  $\tau_{n+1}\in U_n[\frac 1y]\subseteq
U_n[\frac 1{xy}]$, for each $n \in \N$. Hence  $U_{n+1}\subseteq U_n[\frac
1{xy}]$, and    $U\subseteq U_n[\frac 1{xy}]$, and  $U_n[\frac
1{xy}]=U[\frac{1}{xy}].$

The rings $C$ and $D$ are rank-one discrete
valuation domains that are directed unions of two-dimensional
regular local domains. Each of the rings $B_n$  is a  four-dimensional
regular local domain that is a localized polynomial ring  over the
field $k$. Thus $B$ is the directed union of a chain of
four-dimensional regular local domains.
\end{notationremarks}

\begin{theorem} \label{4.2.11t} Assume the setting of  Notation \ref{4.7.n}. Then
the ring   $A$  is a
two-dimensional regular local domain  that birationally dominates
the ring $B$; $A$ has maximal ideal $(x,y)A$ and
completion $\widehat A = k[[x,y]]$. Moreover we have:
\begin{enumerate}
\item The rings $U$ and $B$ are UFDs,
\item  $B$ is a
local  Krull domain with maximal ideal $\n = (x,y)B$,
\item The dimension of $B$ is either 2 or 3, depending on the choice of $\sigma$ and $\tau$,
\item $B$ is Hausdorff
in the topology defined by the powers of $\n$,
\item The $\n$-adic
completion $\widehat B$ of $B$ is canonically isomorphic
to $k[[x,y]]$,
and
\item  The following statements are equivalent:
\begin{enumerate}
\item $B = A$.
\item $B$ is a
two-dimensional regular local domain.
\item  $B$ is Noetherian.
\item Every finitely generated ideal of $B$ is closed in the
$\n$-adic topology on $B$.
\item  Every principal  ideal of $B$ is closed in the
$\n$-adic topology on $B$.
\end{enumerate}
\end{enumerate}
\end{theorem}

\begin{proof} The assertions about  $A$ follow  from
a theorem of Valabrega \cite{V}; see \cite[Proposition~4.13]{powerbook}.
Since $U_0$ has field of fractions $k(x,y,\sigma,\tau)=\mathcal Q(A)$ and
$U_0\subseteq B\subseteq A$, the extension $B \hookrightarrow A$
is birational.
Since $B$ is the directed union of the
four-dimensional regular local domains $B_{n}$
and    $(x,y)B\cap
B_n=(x,y,\sigma_n,\tau_n)B$ for each $n \in \N$,
we see
that $B$ is local with maximal ideal $\n = (x,y)B$.
Since $B$ and $A$ are both dominated by $k[[x,y]]$, it
follows that $A$ dominates $B$.

To prove that $U$ and $B$ are UFDs, we use that $U_n$ is a
polynomial ring over a field and $U_n[\frac 1{xy}]=U[\frac{1}{xy}]$.
Hence  the ring $U[\frac{1}{xy}]$ is a UFD. For each $n \in \N$, the
principal ideals $xU_n$ and $yU_n$ are prime ideals in the
polynomial ring $U_n$. Therefore $xU$ and $yU$ are principal prime
ideals of $U$. Moreover, $U_{xU} = B_{xB}$ and $U_{yU} = B_{yB}$ are
DVRs since each is the contraction to the field $k(x,y,\sigma,\tau)$
of the $(x)$-adic or the $(y)$-adic valuations of $k[[x,y]]$. A
theorem of Nagata \cite[Theorem~6.3, p. 21]{Sam} implies that $U$ is
a UFD; see also \cite[Theorem~2.9 and Fact~2.11]{powerbook}. Since
$B$ is a localization of $U$, the ring $B$ is a UFD. This completes
the proof of items 1 and 2.

Since $B$ is dominated by $k[[x,y]]$, the intersection $\cap_{n=1}^\infty \n^n=(0)$, and so $B$ is Hausdorff in
the topology defined by the powers of $\n$. We have local
injective maps $R \hookrightarrow B \hookrightarrow \widehat R$,
and for each positive integer $n$, we have $\m^nB = \n^n$, $\m^n\widehat R = \widehat \m^n$
and $\widehat \m^n \cap R = \m^n$. Since the natural map
$R/\m^n \to \widehat R/\m^n\widehat R = \widehat R/\widehat \m^n$
is an isomorphism, the  map $R/\m^n \to B/\m^n B = B/\n^n$ is injective and
the  map $B/\n^n  \to \widehat R/\n^n\widehat R = \widehat R/\widehat \m^n$ is surjective.
Since $B/\n^n$ has finite length as an $R$-module, it follows that
$R/\m^n \cong B/\n^n \cong \widehat R/\widehat \m^n$ for each $n \in \N$ and hence
$\widehat B = \widehat R = k[[x,y]]$.
Notice that
 $B$ is a birational extension of
the three-dimensional Noetherian domain $C[y, \tau]$.   The
dimension of $B$ is at most 3 by a theorem of Cohen,
\cite[Theorem~15.5]{M} or \cite[Theorem~2.9]{powerbook}. This
completes the proof of items 3, 4 and 5.

For item 6, since  $A$ is a two-dimensional regular local ring,
$(a) \implies (b)$. Clearly $(b) \implies (c)$.  Since $B$ is local by item~2, and
since the completion of a Noetherian local ring is a faithfully flat
extension,  we have  $(c) \implies (d)$.  It is clear that $(d) \implies (e)$.
To complete the proof of Theorem~\ref{4.2.11t},  it suffices to show that
$(e) \implies (a)$.  Since $A$ birationally dominates $B$, we have $B = A$ if and only if
$bA \cap B = bB$ for every element $b \in \n$.  The
 principal ideal $bB$ is closed in the $\n$-adic topology on $B$ if and
only if $bB = b\widehat B \cap B$.  Also $\widehat B = \widehat A$
 and $bA = b\widehat A \cap A$, for every $b\in B$.  Thus (e) implies, for every $b\in B$,
 $$bB = b\widehat B\cap B=  b\widehat A\cap B= b\widehat A\cap A\cap B=bA \cap B,$$ and so $B = A$.
This completes the proof of Theorem~\ref{4.2.11t}.
\end{proof}

Depending on the choice of $\sigma$ and $\tau$, the ring $B$ may
fail to  be Noetherian. Example~\ref{4.7.13} shows that in the
setting of Theorem~\ref{4.2.11t} the ring $B$ can be strictly
smaller than $ A := k(x,y,\sigma,\tau) \cap k[[x,y]]$.

\begin{example} \label{4.7.13} Using the setting of Notation~\ref{4.7.n},
let    $\tau \in k[[y]]$ be defined to be $\sigma(y)$, that is, set
$b_i := a_i$ for every  $i \in \N$.
 We  then have that
$\theta:=\frac{\sigma - \tau}{x-y}  \in A$. Indeed,
$$
\sigma -  \tau = a_1(x-y) + a_2(x^2 - y^2) + \cdots + a_n(x^n - y^n) + \cdots,
$$
and so  $\theta =\frac{\sigma - \tau}{x-y} \in k[[x, y]] \cap k(x,y,  \sigma, \tau) = A$.
 As a specific  example, one may take  $k := \Q$ and
set $\sigma :=e^x-1$ and  $\tau :=e^y-1$. The ring $B$ is a
localization of the ring $U := \bigcup_{n \in \N}k[x,y,\sigma_n,
\tau_n]$.

\begin{claim} \label{4.7.13.1} The element $\theta$ is not in $B$.
\end{claim}
\begin{proof}
If $\theta$ is an element of $B$, then
$$\sigma-\tau  ~\in ~ (x-y)B ~\cap ~U ~= ~ (x-y)U.
$$
Let $S:= k[x,y,\sigma, \tau]$ and let $U_n := k[x,y,\sigma_n, \tau_n]$ for each positive integer $n$.
We have $$
U ~~= ~~\bigcup_{n\in\N}  U_n ~\subseteq ~ S[\frac{1}{xy}] ~ \subset ~  S_{(x-y)S},
$$
where the last inclusion is because $xy \not\in (x-y)S$.  Thus $\theta \in B$ implies that
$$\sigma-\tau\in (x-y)S_{(x-y)S} ~\cap ~S ~ = ~(x-y)S,$$
but this contradicts the fact that $x,y, \sigma, \tau$ are algebraically
independent over $k$, and thus $S$ is a polynomial ring over $k$ in $x,y, \sigma, \tau$.
\end{proof}

Therefore $\frac{\sigma - \tau}{x-y} \not\in B$, and so $B
\subsetneq A$   and $(x-y)B \subsetneq (x-y)A \cap B$. Since an
ideal of $B$ is closed in the $\n$-adic topology if and only if the
ideal is contracted from $\widehat B$  and since $\widehat B =
\widehat A$, the principal ideal  $(x-y)B$  is not closed in the
$\n$-adic topology on $B$. Using Theorem~\ref{4.2.11t}, we conclude that $B$ is  a
non-Noetherian three-dimensional local Krull domain  having
a two-generated maximal ideal such that $B$ birationally dominates a
four-dimensional regular local domain. In this connection, also  see \cite[Example~8.11]{powerbook}.
\end{example}

The setting of (\ref{4.7.n}) is balanced in $x$ and $y$ in the sense that the
roles of $x$ and $y$ are interchangeable.  A more truly iterative process
is described in Setting~\ref{4.7.nn}.

\begin{setting}\label{4.7.nn}
Let $k$ be a field, let $x$  be an
indeterminate over $k$,  and let
$$
\sigma ~:= ~ \sum_{i=1}^\infty a_ix^i \in k[[x]] \qquad \ {\text{\rm with each}} \qquad \
 a_i ~\in ~ k
$$
be a formal power series that is algebraically independent over the
field $k(x)$. As in Example~\ref{easy},
let $\sigma_n$ be the $n^{\text{th}}$ endpiece of
$\sigma$ and define
$$
C_n ~:=~ k[x,\sigma_n]_{(x,\sigma_n )} \quad \text{ and } \quad
C := k(x,\sigma) \cap k[[x]]= \varinjlim(C_n) = \bigcup_{n=1}^\infty C_n.
$$
Let $y$ be an indeterminate over $C$ and let
$$
\tau ~ := ~ \sum_{i=1}^\infty b_iy^i \in C [[y]] \qquad \ {\text{\rm with each}} \qquad \
 b_i ~\in ~ C
$$
be a formal power series that is algebraically independent over $C[y]$.  Notice that as
a special case we may have each $b_i \in k$.  Let $\tau_n$ be the $n^{\text{th}}$ endpiece of
$\tau$ and define
\begin{equation} \aligned
  U_{n} &:= ~k[x,y,\sigma_n,\tau_n],   \quad U := \varinjlim U_{n} = \bigcup_{n=1}^{\infty} U_n; \\
   B_{n} &:= ~{k[x,y,\sigma_n,\tau_n]}_{(x,y,\sigma_n , \tau_n )}\quad B :=
   \varinjlim (B_{n}) = \bigcup_{n=1}^{\infty} B_n;\\
   D_{n} &:= ~C[y, \tau_n]_{(x,y, , \tau_n )}\quad D :=
   \varinjlim (D_{n}) = \bigcup_{n=1}^{\infty} D_n;\\
   A & := k(x,y,\sigma,\tau )\,
\cap \, k[[x,y]].
\endaligned \tag{\ref{4.7.nn}.0}
\end{equation}
Notice that each $U_n \subset k[[x, y]]$ and $U_n$ is a polynomial ring in
$x, y, \sigma_n, \tau_n$ over the field $k$.   Since $C \subset B$, we have $B = D$.
\end{setting}

\vskip 3pt

\begin{remark}  \label{7.6.1s}   Let the notation be as in Setting~\ref{4.7.nn}.
For certain choices of $\sigma$ and $\tau$,
the ring $B$ is  Noetherian  with $B = A$.  Let $k$ be the field $\Q$ of rational numbers.
Thus  $R := \Q[x,y]_{(x,y)}$
is  the localized polynomial ring in the variables $x$ and $y$,
and the completion $\widehat R$ of $R$ with respect to its  maximal
ideal $\m:=(x,y)R$ is  $\widehat R  = \Q[[x,y]]$, the formal power
series ring in  $x$ and $y$.  Let $\sigma := e^x - 1 \in \Q[[x]]$, and
$C := \Q[[x]] \cap \Q(x, \sigma)$. Thus $C$ is an excellent DVR with maximal ideal $xC$,
and $T := C[y]_{(x,y)C[y])}$ is an excellent countable two-dimensional regular
local ring with maximal ideal $(x,y)T$ and
with $(y)-$adic completion $C[[y]]$.  The UFD $C[[y]]$  has  maximal ideal $\n=(x,y)$.
Since  $T$ is countable, Theorem~\ref{exprimli} implies that there exists
$\tau \in  C[[y]]$ that is primarily limit-intersecting in $y$ over $R$.
Hence
for this choice of $\sigma \in \Q[[x]]$
and $\tau \in C[[y]]$,  we have $A = \Q(x, y, \sigma, \tau) \cap C[[y]]$ is
Noetherian and equal to its approximation domain  $D = B$.

To fit the setting of  Notation~\ref{4.7.n} with $k = \Q$, one wants $\tau \in \Q[[y]]$
rather than $\tau \in C[[y]]$. An example with this more restrictive property is
given in \cite[Example~7.15]{powerbook}.
\end{remark}

\subheading{Weakly flat extensions that are not flat.}

Let $d$ be an integer with $d \ge 2$.
We obtain in Theorem~\ref{wftheorem} extensions that satisfy LF$_{d-1}$ but do
not satisfy LF$_d$; see Definition~\ref{wfetc}.\ref{Lfd12}. Thus we obtain
examples where the intersection domain $A$ is equal to its approximation
domain $B$, but $A$ is not Noetherian.

\begin{theorem} \label{wftheorem}
Let $(R,\m)$ be a countable excellent normal local domain.  Assume that  $\dim R = d+1 \ge 3$,
 that $(x_1, \ldots, x_d, y)R$   is an $\m$-primary ideal, and that
$R^*$ is  the $(y)$-adic  completion of $R$.
 Then there exists $f \in yR^*$ such that $f$ is
algebraically independent over $R$ and the map $\varphi: R[f] \longrightarrow R^*[1/y]$
is weakly flat but not flat. Indeed, $\varphi$  satisfies
LF$_{d-1}$, but fails to satisfy LF$_d$. Thus the intersection domain
$A := \mathcal Q(R[f]) \cap R^*$
is equal to its approximation domain $B$, but $A$ is not Noetherian.
\end{theorem}

\begin{proof}
By Theorem~\ref{exnprimli}, there exist elements $\tau_1, \ldots, \tau_d \in yR^*$ that
are primarily limit-intersecting in $y$ over $R$. Let
$$
f ~ := ~ x_1\tau_1~ + \cdots +~ x_d\tau_d.
$$
Using that $\tau_1, \ldots, \tau_d$ are algebraically independent over $R$, we
regard $f$ as a polynomial in the polynomial ring $T := R[\tau_1, \ldots, \tau_d]$.
Let $S := R[f]$.
For $Q \in \Spec R^*[1/y]$
and $P := Q \cap T$, consider the composition $\varphi_Q$
$$
S ~ \longrightarrow ~ T_P ~  \longrightarrow  ~R^*[1/y]_{Q}.
$$
Since $\tau_1, \ldots, \tau_d$ are primarily limit-intersecting in
$y$ over $R$, the map $T \hookrightarrow R^*[1/y]$ is flat.
Thus the map $\varphi_Q$ is flat if and only if
the map $S \longrightarrow T_P$ is flat.
Let $\p := P \cap R$.

Assume that $P$ is a minimal prime of $(x_1, \ldots, x_d)T$. Then
$\p$ is a minimal prime of $(x_1, \ldots, x_d)R$. Since $T$ is a polynomial
ring over $R$, we have $P = \p T$ and $\hgt(\p) = d = \hgt P$. Notice
that $(\p, f)S = P \cap S$ and $\hgt (\p, f)S = d+1$.  Since a flat extension
satisfies the going-down property, the map $S \longrightarrow T_P$ is not flat.
Hence $\varphi$ does not satisfy LF$_d$.

Assume that $\hgt P \le d-1$.  Then $(x_1, \ldots, x_d)T$ is not contained in $P$.
Hence $(x_1, \ldots, x_d)R$ is not contained in $\p$. Consider the sequence
$$
S ~=~ R[f] ~ \hookrightarrow ~ R_{\p}[f] ~\overset{\psi}\longrightarrow ~ R_{\p}[\tau_1,
\ldots, \tau_d] ~
\hookrightarrow ~T_P,
$$
where the first and last injections are localizations. Since the nonconstant coefficients of $f$
generate the unit ideal of $R_{\p}$, the map $\psi$ is flat; see \cite[Theorem~11.20]{powerbook}.
Thus $\varphi$ satisfies LF$_{d-1}$.

 We conclude that the intersection domain $A = R^* \cap \mathcal Q(R[f])$ is equal to
its approximation domain $B$ and is not Noetherian.
\end{proof}

We describe a  specific example of Theorem~\ref{wftheorem}:

\begin{example} \label{wfexample}
Let $d$ be an integer with $d \ge 2$ and let $x_1, \ldots, x_d, y$ be
indeterminates over a countable  field $k$. Let $R$ be the localized polynomial ring
in the variables $x_1, \ldots, x_d, y$ and let $R^*$ be the $(y)$-adic
completion of $R$. Thus
$$
R~= ~ k[x_1, \ldots, x_d, y]_{(x_1, \ldots, x_d, y)} \quad \text{ and } \quad
R^* ~=~ k[x_1, \ldots, x_d]_{(x_1, \ldots, x_d)}[[y]].
$$
As in Theorem~\ref{wftheorem}, there exist elements $\tau_1, \ldots, \tau_d \in yR^*$ that
are primarily limit-intersecting in $y$ over $R$,  and we consider $ f ~ := ~ x_1\tau_1~ + \cdots +~ x_d\tau_d.$  By Theorem~\ref{wftheorem}, the map $S \longrightarrow R^*[1/y]$  satisfies LF$_{d-1}$, but does not
satisfy LF$_d$. Thus the intersection domain $A = R^* \cap \mathcal Q(R[f])$ is equal to
its approximation domain $B$ and is not Noetherian.

\end{example}

\end{document}